\documentclass[a4paper,10pt]{amsart}

\newtheorem*{prop*}{Proposition}
\newtheorem*{thm*}{Theorem}
\newtheorem*{lem*}{Lemma}

\newcommand{\R}{{\mathbb R}}

\title[On the connectivity of the hyperbolicity region]{On the connectivity of the hyperbolicity region of irreducible polynomials}
\author{Mario Kummer}
\address{Max-Planck-Institute for Mathematics in the Sciences, Leipzig, Germany} 
\email{kummer@mis.mpg.de}

\begin{document}

\begin{abstract}
 We give an elementary proof for the fact that an irreducible hyperbolic polynomial has only one pair of hyperbolicity cones.
\end{abstract}
\maketitle

\section{Introduction}
A homogeneous polynomial $h \in \R[x_1,\ldots,x_n]$ is \textit{hyperbolic with respect to} $e \in \R^{n+1}$ if the univariate polynomial $h(te+v)$ has only real roots for all $v \in \R^{n}$ and if $h(e) \neq 0$. The notion of hyperbolicity is transfered to the projective zero locus of $h$ in the obvious way. Hyperbolic polynomials were first studied in the context of partial differential equations but later they appeared also in many other areas like combinatorics and optimization. For a given homogeneous polynomial $h$ one can consider the set $\Lambda_h$ of all points with respect to which $h$ is hyperbolic. Lars G{\aa}rding \cite{gar} showed that $\Lambda_h$ is the union of several connected components of $\{a\in\R^n:\,h(a)\neq0\}$ all of which are convex cones. Interest in these convex cones comes, among others, from convex optimization since these are the feasible sets of \textit{hyperbolic programs} \cite{Gu97}. It follows directly from the definitions that we have $-\Lambda_h=\Lambda_h$. In particular, connected components of $\Lambda_h$ come in pairs. Under the assumption that $\nabla h(a)\neq0$ for all $0\neq a\in\R^n$ there is at most one such pair by \cite[Thm. 5.2]{HV07}. In general, $\Lambda_h$ can have many connected components. For example a product of linear forms is hyperbolic with respect to any point where it is not zero. Another example is the quartic polynomial $p=(x_1^2+x_2^2-2\cdot x_3^2)\cdot(2\cdot x_1^2-x_2^2-x_3^2)$  which has no linear factor. One easily checks that $\Lambda_p$ has four connected components. It seems to be consensus among experts that if $h$ is irreducible, then the set $\Lambda_h$ should have at most one pair of connected components.
However, we are not aware of any proof for that statement in the literature. This note is intended to give a reference (with proof) for that statement. We were originally motivated by the questions considered in \cite{jorgens2017}.

The above example of the quartic polynomial $p$ without a linear factor suggests that there might be no easy proof for this statement only relying on convex geometric arguments since all connected components of $\Lambda_p$ are strictly convex. Our proof will rely on elementary properties of real algebraic curves and the theorem of Bertini.

\section{Real projective curves}
Let $X$ be a smooth projective and geometrically irreducible curve defined over the reals. It is a well-known fact going back to Felix Klein \cite[\S 21]{type12} that the set $X(\mathbb{C})\setminus X(\mathbb{R})$ has either one or two connected components. If it has two connected components, say $X_+$ and $X_-$, then $X$ is called a \textit{separating curve}. In that case, an orientation on $X_+$ induces an orientation on $X(\mathbb{R})$. Up to global reversing, this orientation does not depend on the choice of the component and the orientation on it. It is called the \textit{complex orientation} on $X(\mathbb{R})$ and was introduced by V. A. Rokhlin \cite[\S 2.1]{rokhlin} in order to study the topology of real plane curves.

Let $f: X\to\mathbb{P}^1$ be a surjective morphism. We fix one of the two orientations on $\mathbb{P}^1(\mathbb{R})$. On every point $p\in X(\mathbb{R})$ where $f$ is unramified the fixed orientation on $\mathbb{P}^1(\mathbb{R})$ induces an orientation of $X(\mathbb{R})$ locally around $p$. Again, up to global reversing, this orientation does not depend on the choice of the orientation on $\mathbb{P}^1(\mathbb{R})$. We call it the \textit{orientation induced by $f$} on $X(\mathbb{R})$. 

Now assume that $f$ has the property $f^{-1}(\mathbb{P}^1(\mathbb{R}))= X(\mathbb{R})$. Then, since $\mathbb{P}^1(\mathbb{C})\setminus\mathbb{P}^1(\mathbb{R})$ is not connected, the set $X(\mathbb{C})\setminus X(\mathbb{R})$ is also not connected. It consists therefore of two connected components $X_+$ and $X_-$ which are the preimages under $f$ of the upper and lower half-planes respectively, cf. \cite[\S 3.6]{rokhlin}. We call $f$ a \textit{separating morphism} because it is a certificate for $X$ being separating. Note that in that case $f$ is unramified on all of $X(\mathbb{R})$ (see e.g. \cite[Rem. 3.2]{gaba}) and the orientation induced by $f$ is the same as the complex orientation.

\section{Plane hyperbolic curves}
 Now let $Y$ be a possibly singular projective and geometrically irreducible curve defined over the reals. As above we call a morphism $f: Y\to\mathbb{P}^1$ \textit{separating} if $f^{-1}(\mathbb{P}^1(\mathbb{R}))= Y(\mathbb{R})$. Again this induces an orientation on the smooth points of $Y(\mathbb{R})$. On the other hand, the desingularization map $g:\tilde{Y}\to Y$ together with the complex orientation on $\tilde{Y}(\mathbb{R})$ also induces an orientation on the smooth points of $Y(\mathbb{R})$. Up to global reversing, these two orientations are the same. Indeed, $g$ is an isomorphism on the smooth points of $Y$ and the morphism $f\circ g$ is also separating. This gives us the following lemma.
 \begin{lem*}
 Two separating morphisms $f_1, f_2$ from $Y$ induce the same orientations on the smooth points of $Y(\mathbb{R})$ --- up to global reversing.
 \end{lem*}
 Let $C\subseteq\mathbb{P}^2$ be an irreducible projective plane curve defined over the reals. For every point $e\in\mathbb{P}^2(\mathbb{R})$ we consider the linear projection $\pi_e: C\setminus\{e\}\to L$ where $L\subseteq\mathbb{P}^1$ is a line not containing $e$. Let $p\in C(\mathbb{R})$ and $e_1, e_2\in\mathbb{P}^2(\mathbb{R})$ be two points not lying on the tangent space $T$ of $C$ at $p$. Let $L\subseteq\mathbb{P}^2$ be a line not containing $e_1$ and $e_2$ and fix an orientation on $L(\mathbb{R})$. We observe that the orientation on $C(\mathbb{R})$ at $p$ induced by $\pi_{e_1}$ is the same as the one induced by $\pi_{e_2}$ if and only if $e_1$ and $e_2$ are in the same connected component of $\mathbb{P}^2(\mathbb{R})\setminus(L\cup T)$.
 Now assume that $C$ is hyperbolic, i.e. the set $$\Lambda_C=\{e\in\mathbb{P}^2(\mathbb{R}):\,C\textnormal{ is hyperbolic with respect to }e\}$$ is not empty. Then for each $e\in\Lambda_C$, the projection $\pi_e:C\to\mathbb{P}^1$ is separating. Since a separating morphism is not ramified at smooth real points, this means that $\Lambda_C$ does not intersect any line that is tangent to $C$ at a smooth real point.
 \begin{prop*}
  In the above situation the set $\Lambda_C$ is connected.
 \end{prop*}
 \begin{proof}
  Assume for the sake of a contradiction that $\Lambda$ is not connected. Let $\Lambda_1$ and $\Lambda_2$ be two connected components of $\Lambda$ and let $e_i\in\Lambda_i$ for $i=1,2$. Let $p_1$ be a smooth point of $C(\mathbb{R})$ which is in the closure of $\Lambda_1$ and let $T_1$ be the tangent of $C$ at $p_1$. Then by the above discussion we have $\Lambda\subseteq\mathbb{P}^2(\mathbb{R})\setminus T_1\cong\mathbb{R}^2$. Thus, $\Lambda_1$ and $\Lambda_2$ are two disjoint convex open subsets of $\mathbb{R}^2$. We can find a line $G$ through $e_1$ that intersects $\Lambda_2$ and which intersects $C$ transversally and only in smooth points. The linear polynomial that defines a tangent $T_2$ to $C$ at one of the two intersection points of $G$ with the boundary of $\Lambda_2$, lets denote it it by $p_2$, will separate $e_1$ from $\Lambda_2$. Thus $e_1$ and $e_2$ are in distinct connected components of $\mathbb{P}^2(\mathbb{R})\setminus(T_1\cup T_2)$. Now we compare the orientations at $p_1$ and $p_2$ induced by the two morphisms $\pi_{e_i}: C\to T_1$. By the above discussion we see that these induce different orientations at $p_2$ but the same at $p_1$. But this is a contradiction to the conclusion of the lemma.
 \end{proof}

 \section{The general case}
 Now the general case follows from a simple application of the theorem of Bertini.
 \begin{thm*}
  Let $h \in \R[x_1,\ldots,x_n]$ be irreducible and hyperbolic. If $n>2$, the set $\Lambda_h$ of all points with respect to which $h$ is hyperbolic has two connected components.
 \end{thm*}
 \begin{proof}
  Let $V\subseteq\mathbb{P}^{n-1}$ be the hypersurface defined by $h$. Assume for the sake of a contradiction that $$\Lambda_V=\{e\in\mathbb{P}^{n-1}(\mathbb{R}):\,V\textnormal{ is hyperbolic with respect to }e\}$$ is not connected. Then we can find a real linear subspace $E\subseteq\mathbb{P}^{n-1}$ of dimension two that intersects two different connected components of $\Lambda_V$ such that $E\cap V$ is irreducible \cite[Cor. 6.11]{bertini}. But this contradicts the proposition which implies that $\Lambda_V\cap E$ is connected.
 \end{proof}

\bigskip

 \noindent \textbf{Acknowledgements.}
I would like to thank Daniel Plaumann, who encouraged me to write this note, and Eli Shamovich for some comments.


\begin{thebibliography}{Gab06}

\bibitem[Gab06]{gaba}
Alexandre Gabard.
\newblock Sur la repr\'esentation conforme des surfaces de {R}iemann \`a bord
  et une caract\'erisation des courbes s\'eparantes.
\newblock {\em Comment. Math. Helv.}, 81(4):945--964, 2006.

\bibitem[G{\.a}r59]{gar}
Lars G{\.a}rding.
\newblock An inequality for hyperbolic polynomials.
\newblock {\em J. Math. Mech.}, 8:957--965, 1959.

\bibitem[G{\"u}l97]{Gu97}
Osman G{\"u}ler.
\newblock Hyperbolic polynomials and interior point methods for convex
  programming.
\newblock {\em Math. Oper. Res.}, 22(2):350--377, 1997.

\bibitem[HV07]{HV07}
J.~William Helton and Victor Vinnikov.
\newblock Linear matrix inequality representation of sets.
\newblock {\em Comm. Pure Appl. Math.}, 60(5):654--674, 2007.

\bibitem[Jou83]{bertini}
Jean-Pierre Jouanolou.
\newblock {\em Th\'eor\`emes de {B}ertini et applications}, volume~42 of {\em
  Progress in Mathematics}.
\newblock Birkh\"auser Boston, Inc., Boston, MA, 1983.

\bibitem[JT17]{jorgens2017}
Thorsten J{\"o}rgens and Thorsten Theobald.
\newblock Hyperbolicity cones and imaginary projections.
\newblock {\em arXiv preprint arXiv:1703.04988}, 2017.

\bibitem[Kle63]{type12}
Felix Klein.
\newblock {\em On {R}iemann's theory of algebraic functions and their
  integrals. {A} supplement to the usual treatises}.
\newblock Translated from the German by Frances Hardcastle. Dover Publications,
  Inc., New York, 1963.

\bibitem[Roh78]{rokhlin}
V.~A. Rohlin.
\newblock Complex topological characteristics of real algebraic curves.
\newblock {\em Uspekhi Mat. Nauk}, 33(5(203)):77--89, 237, 1978.

\end{thebibliography}
 \end{document}